\documentclass{amsart}
\usepackage{amssymb,xspace}
\usepackage[small]{diagrams}
\usepackage[pdfstartview=FitH]{hyperref}
\def\mat#1{\ensuremath{#1}\xspace}

\def\cA{\mat{\mathbb{A}}}   %affine space

\def\cF{\mat{\mathbb{F}}}
\def\cL{\mat{\mathbb{L}}}   %
\def\cN{\mat{\mathbb{N}}}   %natural numbers
\def\cQ{\mat{\mathbb{Q}}}   %ring of rationals
\def\cC{\mat{\mathbb{C}}}   %complex numbers

\def\cZ{\mat{\mathbb{Z}}}   %ring of integers

\def\lA{\mat{\mathcal{A}}}

\def\lE{\mat{\mathcal{E}}}

\def\lH{\mat{\mathcal{H}}}

\def\lM{\mat{\mathcal{M}}}

\def\lP{\mat{\mathcal{P}}}

\let\Phitemp\Phi \def\Phi{\mat{\Phitemp}}
\let\Psitemp\Psi \def\Psi{\mat{\Psitemp}}
\let\etatemp\eta \def\eta{\mat{\etatemp}}
\def\La{\mat{\Lambda}}
\def\la{\mat{\lambda}}

\let\mutemp\mu\def\mu{\mat{\mutemp}}
\let\nutemp\nu\def\nu{\mat{\nutemp}}
\let\pitemp\pi\def\pi{\mat{\pitemp}}
\def\si{\mat{\sigma}}

\def\om{\mat{\omega}}

\def\al{\mat{\alpha}}
\def\be{\mat{\beta}}
\def\ga{\mat{\gamma}}
\def\Ga{\mat{\Gamma}}

\def\hi{\mat{\chi}}

\let\xitemp\xi \def\xi{\mat{\xitemp}}

\def\mrm@#1{\mat{\mathrm{#1}}}

\def\DMO{\DeclareMathOperator}

\DMO{\Hom}{Hom}
\DMO{\lHom}{\lH\mathit{om}}
\DMO{\Ext}{Ext}
\DMO{\lExt}{\lE\mathit{xt}}
\DMO{\End}{End}
\DMO{\Aut}{Aut}
\DMO{\Fun}{Fun}
\DMO{\Tor}{Tor}
\DMO{\ext}{ext}

\DMO{\Ob}{Ob}
\DMO{\Mor}{Mor}
\DMO{\im}{im}
\DMO{\coim}{coim}
\DMO{\coker}{coker}
\DMO{\Arr}{Arr}
\DMO{\Id}{Id}
\DMO{\add}{add} % splitting of idempotents (karoubinization)
\DMO{\ind}{ind} % category of ind-objects
\DMO{\pro}{pro} % category of pro-objects
\DMO{\Map}{Map} %
\DMO{\Iso}{Iso} %
\DMO{\Isom}{Isom}%
\DMO{\Ind}{Ind}

\DMO{\Presh}{Presh}

\DMO\coalg{Coalg}
\DMO{\Rep}{Rep}
\DMO{\Cor}{Cor}

\DMO{\Mod}{Mod}
\DMO{\rad}{rad}
\DMO{\soc}{soc}
\DMO{\ann}{ann}

\DMO{\Spec}{Spec}
\DMO{\spec}{Spec}
\DMO{\Proj}{Proj}
\DMO{\supp}{supp}
\DMO{\Coh}{Coh}
\DMO{\coh}{Coh}
\DMO{\Qcoh}{QCoh}
\DMO{\QCoh}{QCoh}
\DMO{\Pic}{Pic}
\DMO{\Div}{Div}
\DMO{\ch}{ch}
\DMO{\Hilb}{Hilb}
\DMO{\Fitt}{Fitt}
\DMO{\Quot}{Quot}

\DMO{\Gras}{Gr}
\DMO{\Flag}{Flag}

\DMO{\cone}{cone}
\DMO{\Tw}{Tw}
\DMO{\rank}{rk}
\DMO{\rk}{rk}
\DMO{\codim}{codim}
\DMO{\cov}{cov}
\DMO{\sgn}{sgn}
\DMO{\td}{td}
\DMO{\GL}{GL}
\DMO{\SL}{SL}

\DMO\Der{Der}
\DMO\der{Der}
\DMO\coder{Coder}
\DMO{\diag}{diag}
\DMO{\HMod}{HMod} %the homotopy category of modules over DGC
\DMO{\ad}{ad}
\DMO*{\colim}{colim}
\DMO*{\hocolim}{hocolim}
\DMO*{\holim}{holim}
\DMO{\Ho}{Ho}
\def\Gr{\mathbf{Gr}} %graded

\DMO{\har}{char}
\DMO{\sk}{sk}
\DMO{\cosk}{cosk}
\DMO{\Gal}{Gal}
\DMO{\tr}{tr}
\DMO{\Tr}{Tr}
\DMO{\Sh}{Sh}
\DMO{\Is}{Is} %Isometries
\DMO{\Hol}{Hol} %Holomorphic automorphisms
\DMO{\Lie}{Lie} %Lie algebra of a group
\DMO{\Res}{Res} %restriction
\DMO{\irr}{irr} %
\DMO{\Irr}{Irr} %
\DMO{\Exp}{Exp} %
\DMO{\Log}{Log} %
\DMO{\Pow}{Pow}
\DMO{\mult}{mult} %
\DMO{\height}{ht} %
\DMO{\wt}{wt}
\DMO{\Vect}{Vect}
\DMO{\moda}{mod}

\def\iso{\simeq}
\def\ts{\otimes}

\def\what#1{\mat{\widehat{#1}}}

\def\sb{\subset}

\def\sp{\supset}

\def\xx{\times}

\def\ms{\backslash} %minus set
\def\pser#1{[\![#1]\!]} %formal power series [[#1]]
\def\inv{^{-1}}

\def\tw{^{\rm tw}}

\def\st{^{s}}
\def\sst{^{ss}}

\def\ang#1{\mat{\left\langle #1\right\rangle}}

\def\set#1{\mat{\{ #1\}}}
\def\sets#1#2{\mat{\{ #1 \mid #2\}}}

\def\smat#1{\mat{\left(\begin{smallmatrix}#1\end{smallmatrix}\right)}}

\def\arrowsUsual{
\newarrow{TeXto}----{->}
\newarrow{TeXinto}C---{->}
\newarrow{TeXonto}----{->>}
\newarrow{TeXdashto}{}{dash}{}{dash}{->}
\def\ar{\rightarrow}
\def\emb{\hookrightarrow}
\def\mto{\mapsto}
\def\arr{\rTeXto}
\def\embb{\rTeXinto}
\newarrow{Eq}=====
    }

\newif\ifukr\ukrfalse
\newif\ifrus\rusfalse
\newif\ifger\gerfalse

\def\theorems{
\newcounter{nthr} %auxiliary counter
\numberwithin{nthr}{section}
\newtheorem{thr}[nthr]{Theorem}
\newtheorem{prp}[nthr]{Proposition}
\newtheorem{lmm}[nthr]{Lemma}
\newtheorem{crl}[nthr]{Corollary}
\newtheorem{clm}[nthr]{Claim}
\newtheorem{conj}[nthr]{Conjecture}
\theoremstyle{definition}
\newtheorem{dfn}[nthr]{Definition}
\newtheorem{rmr}[nthr]{Remark}
\newtheorem{exm}[nthr]{Example}}

\def\br{\linebreak}
\def\ub#1{\mat{\overline{#1}}}  %upper bar
\def\eq{\equation}\def\endeq{\endequation}

\arrowsUsual\theorems
\begin{document}

\title[]{Poincar\'e polynomials of moduli spaces 
of stable bundles over curves}%
\author{Sergey Mozgovoy}
\email{mozgov@math.uni-wuppertal.de}
%\date{5.11.07}
%\institute{Department of Mathematics and Informatics,
%University of Wuppertal,
%Gau\ss straße 20,
%42119 Wuppertal,
%Germany}

\begin{abstract}
Given a curve over a finite field, we compute the number
of stable bundles of not necessarily coprime rank and degree over it.
We apply this result to compute the virtual 
Poincar\'e polynomials of the moduli spaces of stable bundles
over a curve. A similar formula for the virtual 
Hodge polynomials and motives is conjectured.
\end{abstract}\maketitle
\section{Introduction}
Let $X$ be a smooth projective curve of genus $g$ over \cC.
Let $\lM(n,d)$ be the moduli space of stable vector bundles
of rank $n$ and degree $d$ on $X$.
The problem of computation of the Betti numbers of $\lM(n,d)$
has attracted the interest of many people. In the case of coprime $n$ and $d$,
a recursive formula for the Poincar\'e polynomial of $\lM(n,d)$
was given by Harder, Narasimhan, Desale and Ramanan \cite{HarNar1,DesRam1} by using
the Weil conjectures. The same recursive formula
was obtained later by Atiyah and Bott \cite{AtBot1} using the gauge theory.
Another method to prove the same recursive formula
was found by Bifet, Ghione and Letizia \cite{BGL1}.
The method of Atiyah and Bott was adapted later by Earl and Kirwan
\cite{EarlKirw1}
to give a similar
recursive formula for the Hodge polynomial of $\lM(n,d)$.
The method of Bifet, Ghione and Letizia was adapted
by del Ba{\~n}o \cite{Bano1} to give an analogous formula for the motive
of $\lM(n,d)$. 
The recursive formula for the Poincar\'e polynomials was solved
by Zagier \cite{Zagier1}. As a result, he got a nice explicit formula,
which is much simpler than the initial recursion (e.g.\ it does not contain
infinite sums in contrast to the recursive formula).
The formula of Zagier was further generalized by
Laumon and Rapoport \cite{LaumRap1}.
The formula of Zagier can be easily adapted for the cases of Hodge polynomials
and motives
(see \cite{Bano1} and Section \ref{sec:conj}).

In the case of not necessarily coprime rank and degree only
some partial results are known. The Betti and Hodge numbers of
$H^i(\lM(n,d))$ were computed for
$i<2(n-1)g-(n-1)(n^2+3n+1)-7$
by Arapura and Sastry \cite{ArapSast1} and for
$i<2(n-1)(g-1)$
by Dhillon \cite{Dhil1}.

In this paper we go in a somewhat different direction and determine
the virtual Poincar\'e polynomial
(see \cite{DanKhov} and Remark \ref{rmr:poincare of variety})
of $\lM(n,d)$ for arbitrary $n$ and $d$.
We use the approach by Harder, Narasimhan, Desale and Ramanan
and compute first the number of points of the corresponding
moduli space over finite fields and then use the Weil conjectures
to obtain the virtual Poincar\'e polynomial of this moduli space.
Given a curve $X$ over a finite field $k$,
a recursive formula has been proved in \cite{HarNar1,DesRam1} for a
weighted sum
$$r_{(n,d)}(k)=\sum_{M\in\lE_X\sst(n,d)}\frac1{\#\Aut M},$$
where $\lE_X\sst(n,d)$ denotes the set of representatives of
isomorphism classes of semistable sheaves over $X$ having rank $n$
and degree $d$. We show that the numbers $r_{n,d}(K)$ (for all
possible $n$, $d$ and finite field extensions $K/k$) can be used
to compute the numbers $a_{(n,d)}(k)$ of absolutely stable (see
Section \ref{subsec:stable}) sheaves over $X$ of rank $n$ and
degree $d$. The last number is precisely the number of the
$k$-rational points of the moduli space $\lM(n,d)$.

The same approach was used in \cite{MR1} to relate the weighted
sums over semistable representations of a quiver over finite
fields to the numbers of absolutely stable representations. There
the analogues of $r_{(n,d)}(\cF_{q})$ and $a_{(n,d)}(\cF_{q})$
(here $\cF_q$ is a finite field extension of $k$) are rational
functions (respectively, polynomials) in $q$ and it is possible to
apply the machinery of \la-rings to the usual \la-ring structure
on $\cQ(q)$ (see Section \ref{subsec:la-rings}) to get the
relation between the corresponding rational functions and
polynomials. In the case of sheaves over a curve, the numbers
$r_{(n,d)}(\cF_q)$ and $a_{(n,d)}(\cF_q)$ are not, in general,
rational functions in $q$ and we cannot use the formalism of
\la-rings. To pass around this problem, we introduce a new
\la-ring, called the ring of c-sequences (see Section
\ref{sec:c-seq}) such that the systems of rational numbers
$r_{(n,d)}=(r_{(n,d)}(K))_{K/k}$ and
$a_{(n,d)}=(a_{(n,d)}(K))_{K/k}$, where $K$ runs through the
finite field extensions of $k$, are elements of this ring.
Actually, for any algebraic variety $X$ over $k$, the system
$(\#X(K))_{K/k}$ is a c-sequence. We use the \la-ring of
c-sequences as an analogue of the Grothendieck ring of motives in
the case of a finite field. We prove a formula relating the
c-sequences $r_{(n,d)}$ and $a_{(n,d)}$ using the formalism of
\la-rings in a similar way as it was done for representations of
quivers in \cite{MR1}.

With any c-sequence, we associate its Poincar\'e function.
For any algebraic variety over $k$, its virtual Poincar\'e
polynomial coincides with the Poincar\'e function
of the associated c-sequence. The relation between the
c-sequences $r_{(n,d)}$ and $a_{(n,d)}$ induces the relation
between the corresponding Poincar\'e functions. The virtual
Poincar\'e polynomial of $\lM(n,d)$ is precisely the Poincar\'e
function of the c-sequence $a_{(n,d)}$. The Poincar\'e function
of $r_{(n,d)}$ can be computed using the recursive formula
of Harder, Narasimhan, Desale and Ramanan or using the explicit
formula due to Zagier.

Concerning the virtual Hodge polynomials (see \cite{DanKhov}),
we give a conjectural
formula that relates the virtual Hodge polynomials of $\lM(n,d)$
and certain rational functions, that should be understood as
analogues of the Poincar\'e function of $r_{(n,d)}$. These rational
functions where used by Earl and Kirwan \cite{EarlKirw1} to compute
the Hodge polynomial of $\lM(n,d)$ for coprime $n$ and $d$.

We also give a conjecture for the motive of $\lM(n,d)$.
In \cite{BehrDhil1}, Behrend and Dhillon found a formula
for the motives of the stacks of all bundles on a curve.
One can relate the motives of the
stacks of semistable bundles (which are analogues of $r_{(n,d)}$)
and the motives of the stacks of all bundles using the
Harder--Narasimhan filtration and get a formula
similar to the recursive formula of Harder, Narasimhan,
Desale and Ramanan. This was done, using a different language,
by del Ba{\~n}o \cite{Bano1}. The solution of the recursive formula
due to Zagier can also be applied here. 
We give a conjectural formula relating the motives of stacks
of semistable bundles and motives of moduli spaces
of stable bundles. To formulate it we use the standard \la-ring
structure on the Grothendieck ring of motives, where
the \si-operations are given by symmetric products.

The paper is organized as follows. In Section \ref{sec:prelim} we
gather some preliminary material: basic operations in complete
\la-rings,  the Galois descent for the sheaves on a variety over a
finite field, indecomposable and stable sheaves, some results on
the Hall algebra of the category of sheaves on a curve over a
finite field. In Section \ref{sec:HN rec} we discuss solutions of
the recursive formulas that occur from Harder-Narasimhan
filtrations. In Section \ref{sec:c-seq} we introduce the \la-ring
of c-sequences and prove its basic properties. In Section
\ref{sec:semist} we discuss formulas for c-sequences $r_{(n,d)}$
and its Poincar\'e functions. In Section \ref{sec:stable} we prove
formulas relating the c-sequences $r_{(n,d)}$ and $a_{(n,d)}$ and
their Poincar\'e functions. As a corollary, we give a formula for
the virtual Poincar\'e polynomials of the moduli spaces of stable
vector bundles on a curve. In Section \ref{sec:conj} we give
conjectural formulas for the virtual Hodge polynomials and for the
motives of the moduli spaces of stable vector bundles on a curve.

I would like to thank Markus Reineke for many helpful and
encouraging discussions. The suggestion that our methods
from \cite{MR1} could be applied in the case of stable
bundles is due to him.

\section{Preliminaries}\label{sec:prelim}
\subsection{Basic operations in complete
\texorpdfstring{\la}{lambda}-rings}\label{subsec:la-rings}
Any \la-ring $R$ (see e.g.\ \cite{Getz1,M2}) has three families of operations:
\la-operations, \si-operations and Adams operations. Our \la-rings
are always algebras over \cQ. Under this condition, the \la-ring
structure is uniquely determined by any of the above three families of
operations. We will usually use the Adams operations $\psi_n:R\ar R$, $n\ge1$,
to describe the \la-ring structure.
Given a \la-ring $R$, we endow the ring $R[x_1,\dots,x_r]$ with
a structure of a \la-ring by the formula
$$\psi_n(ax^\al)=\psi_n(a)x^{n\al},\qquad a\in R,\al\in\cN^r.$$
In the same way we endow the ring $R\pser{x_1,\dots,x_r}$
with a structure of a \la-ring. 
Let $R$ be a \la-ring and let $S$ be a multiplicatively closed subset
of $R$ that is closed under Adams operations. Then we can endow 
the localization $S\inv R$ with a structure of a \la-ring by the
formula
$$\psi_n(a/s)=\psi_n(a)/\psi_n(s),\qquad a\in R,s\in S.$$
If otherwise not stated, we always endow our rings with \la-ring
structures in the above described way.

A complete filtered \la-ring $R$ is a complete filtered ring 
endowed with a \la-ring structure 
respecting the filtration, that is $\psi_n(F^mR)\sb F^{nm}R$,
where $R\sp F^1R\sp F^2R\sp\dots$ is a filtration of $R$.
For example, if $R$ is a \la-ring, then $R\pser{x_1,\dots x_r}$
is a complete filtered \la-ring with the filtration induced by total grading.
Let $R$ be a complete filtered \la-ring. 
Define the map $\Exp:F^1R\ar1+F^1R$ by the formula
$$\Exp(f)=\sum_{k\ge0}\si_k(f)=\exp\Big(\sum_{k\ge1}\frac1k\psi_k(f)\Big).$$
We have $\Exp(f+g)=\Exp(f)\Exp(g)$, for $f,g\in F^1R$. The map
$\Exp$ has an inverse $\Log:1+F^1R\ar F^1R$ (see \cite{Getz1,M2}) given by
$$\Log(f)=\sum_{k\ge1}\frac{\mu(k)}k\psi_k(\log(f)),$$
where $\mu$ is a M\"obius function.
Define the map $\Pow:(1+F^1R)\xx R\ar1+F^1R$  by the formula
$$\Pow(f,g)=\Exp(g\Log(f)),\qquad f\in1+F^1R, g\in R.$$
Define the usual power map by $f^g=\exp(g\log(f))$.

\begin{lmm}[see {\cite[Lemma 22]{M2}}]\label{lmm:power}
Let $R$ be a complete filtered \la-ring, $f\in1+F^1R$ and $g\in R$.
Define the elements $g_k\in R$, $k\ge1$, inductively by the formula
$\sum_{k\mid n}kg_k=\psi_n(g)$, $n\ge1$. Then we have
$$\Pow(f,g)=\prod_{k\ge1}\psi_k(f)^{g_k}.$$
\end{lmm}
  
For any $m\ge0$, define $[\infty,m]_v=\prod_{i=1}^m(1-v^i)\inv\in\cQ(v)$.

\begin{lmm}[Heine formula, see \cite{M4}]\label{lmm:heine}
We have
$$\Exp\Big(\frac x{1-v}\Big)=\sum_{m\ge0}[\infty,m]_vx^m$$
in $\cQ(v)\pser x$.
\end{lmm}  
  
\subsection{Chern character}\label{subsec:chern}
Let $X$ be a smooth projective curve over a field $k$. 
The map 
$$\ch:K_0(\Coh X)\ar\cZ^2$$ 
given by $[F]\ar (\rk F,\deg F)$ is called the Chern character.
By the Riemann-Roch formula \cite[Ch. 15]{Fulton1}, we have
\begin{align*}
\hi(F,G)=&\dim\Hom(F,G)-\dim\Ext^1(F,G)\\
=&\rk F\deg G-\deg F\rk G+(1-g)\rk F\rk G.
\end{align*}
For any elements $\al=(n,d)$, $\be=(n',d')$ in $\cZ^2$, define
$$\ang{\al,\be}=nd'-dn'+(1-g)nn'=\al C\be^t,$$
where $C=\smat{1-g&1\\-1&0}$. Then 
$$\hi(F,G)=\ang{\ch F,\ch G}.$$

\subsection{Galois descent}
Let $X$ be an algebraic variety over a finite field $k$ and
let $K/k$ be a finite field extension.
The group $\Ga=\Gal(K/k)$ acts on $X_K=X\ts_kK$ in a natural way.
For any coherent sheaf $M$ over $X_K$ and for any $\si\in\Ga$,
there is defined a direct image sheaf $\si_*M$.
A $\Ga$-equivariant coherent sheaf $M$ on $X_K$ is a coherent sheaf 
endowed with a system of isomorphisms $\si_M:M\ar\si_*M$, $\si\in\Ga$,
satisfying $\si (\tau_M)\circ\si_M=(\si\tau)_M$ for any $\si,\tau\in\Ga$. 
Note that the map $f:M\ar\si_*M$ of coherent sheaves is given
by the map $\ub f:M\ar M$ of abelian sheaves, satisfying 
locally $\ub f(am)=\si(a)\ub f(m)$.
It follows that if $\si\in\Ga$ is a generator, then the structure
of a $\Ga$-equivariant sheaf on $M$ is given by the map $f:M\ar\si_*M$
such that $(\ub f)^{\#\Ga}=1$.
We denote by $\Coh_\Ga X_K$ the category
of \Ga-equivariant coherent sheaves on $X_K$.
For any coherent sheaf $M$ on $X$,
the sheaf $M_K=M\ts_kK$ on $X_K$ is a \Ga-equivariant sheaf
in a canonical way. 
The following result is a version of the Galois descent
(see, e.g., \cite{SGA1})

\begin{prp}
The functor $\Coh X\ar\Coh_\Ga X_K$, sending a sheaf $M$ to 
the \Ga-equivariant sheaf $M_K$
is an equivalence of categories. The inverse functor is given
by $N\mto N^\Ga$, the \Ga-invariant part of $N$.
\end{prp}

\subsection{Indecomposable sheaves}
Let $X$ be a projective scheme over a finite field~$k$.
Then all the homomorphism spaces between coherent sheaves
on $X$ are finite dimensional and the category $\Coh X$ is
Krull-Schmidt \cite{Atiyah1}. We say that a sheaf $M\in\Coh X$ is absolutely
indecomposable if, for any finite field extension $K/k$,
the sheaf $M_K$ on $X_K$ is indecomposable. For any
sheaf $M\in\Coh X$, we define the splitting field of a sheaf $M$ 
to be the minimal finite field extension $K/k$ such that all the 
indecomposable direct summands of $M_K$ are absolutely indecomposable
(it will be shown that such a field is unique).
Given a finite field extension $K/k$ and a sheaf $M\in\Coh X_K$, 
we define the minimal field of definition of $M$ to be the 
minimal field extension $L/k$ contained in $K$ 
such that there exists a sheaf $M'\in\Coh X_L$ with $M'_K\iso M$
(it will be shown that such a field is unique).
For any coherent sheaf $M\in\Coh X_K$, we define
$k(M)=\End M/\rad(\End M)$.  

\begin{lmm}
For any finite field extension $L/K$,
we have $k(M_L)=k(M)\ts_KL$.
\end{lmm}
\begin{proof}
It is enough to show that, for any finite-dimensional algebra
$A$ over $K$, we have $\rad(A_L)=(\rad A)_L$.
This follows from \cite[Cor. 7.2.2]{BourAlg8} as $L/K$
is a separable field extension.
\end{proof}

\begin{lmm}
Let $M,N\in\Coh X$ and assume that there exists
some finite extension $K/k$ such that $M_K\iso N_K$.
Then $M\iso N$.
\end{lmm}
\begin{proof}
Consider the finite etale morphism $\pi:X_K\ar X$.
Then $\pi_*(M_K)\iso M^{n}$ and
$\pi_*(N_K)\iso N^{n}$, where $n=[K:k]$.
As $\Coh X$ is a Krull-Schmidt category, we get $M\iso N$.
\end{proof}

This lemma implies that any two $\Gal(K/k)$-equivariant
structures on the sheaf on $X_K$ descend
to isomorphic sheaves on $X$. The following proposition 
should be compared with \cite[Lemma 3.4]{Kac2}

\begin{prp}\label{prp:indecomp}
Let $K/k$ be a finite field extension and let $M\in\Coh X_K$ be a coherent sheaf.
Then
\begin{enumerate}
	\item $M$ is indecomposable if and only if $k(M)$ is a field.
	\item $M$ is absolutely indecomposable if and only if $k(M)=K$.
	\item If $M$ is indecomposable then the splitting field of $M$ is unique and
	equals $k(M)$.
	\item If $M$ is indecomposable then the  minimal field of definition of 
	$M$ is unique and equals $K^\Ga$, where 
	$\Ga=\sets{\si\in\Gal(K/k)}{\si_*M\iso M}$.	
	\item If $N\in\Coh X$ is indecomposable and $K$ is its splitting field, then
	there exists an absolutely indecomposable sheaf $M\in\Coh X_K$ with a minimal
	field of definition $K$ 
	such that $N_K\iso\bigoplus_{\si\in\Gal(K/k)}\si_* M$ 
	and both $\Gal(K/k)$-equivariant sheaves descend to $N$.
	\item If $M$ is indecomposable and has the minimal field of definition $K$,
	then the $\Gal(K/k)$-equivariant sheaf $\bigoplus_{\si\in\Gal(K/k)}\si_* M$
	descends to the indecomposable sheaf over $X$ having the splitting field
	$k(M)$. 
\end{enumerate}
\end{prp}
\begin{proof}
1) If $M$ is indecomposable then $\End M$ is a local algebra and $k(M)$ is
  a division algebra over a finite field $K$. Thus, $k(M)$ is a field.
  If $M\iso\oplus_{i=1}^nN_i^{r_i}$, where $M_i$ are pairwise non-isomorphic 
	sheaves then $k(M)\iso\prod_{i=1}^nM(r_i,k(N_i))$. 
	This is a field if and only if $n=1$ and $r_1=1$.\\
2) If $k(M)=K$ and $L/K$ is a finite field extension then 
	$k(M_L)\iso k(M)\ts_KL\iso L$ and therefore $M_L$ is indecomposable.
  If $M$ is absolutely indecomposable and $k(M)=L$, then 
  $k(M_L)\iso L\ts_KL\iso\prod_{i=1}^{[L:K]}L$. As $M_L$ is indecomposable,
  we have $[L:K]=1$.\\
3) Let $L=k(M)$ and let $L'$ be some splitting field of $M$.
	It is known \cite[Proposition~8.3]{BourAlg8} that 
	$k(M_{L'})\iso L\ts_KL'\iso\prod_{i=1}^n L_i$ 
	is a product of fields. The corresponding decomposition
	of unit into the sum of idempotents can be lifted
	to $\End(M_{L'})$ and we get a decomposition 
	$M_{L'}=\bigoplus_{i=1}^nM_i$, where $k(M_i)\iso L_i$. It follows that the sheaves
	$M_i$ are indecomposable and therefore absolutely indecomposable.
	This implies that $L_i=L'$. It is known \cite[Proposition~8.3]{BourAlg8}
	that every $L_i$ is a composite field of the fields $L$ and $L'$.  
	Therefore $L$ is contained in $L_i=L'$.
	To show that $L$ is itself a splitting field, we 
	note that $k(M_{L})\iso L\ts_KL\iso\prod_{i=1}^{[L:K]}L$
	by \cite[Proposition~8.4]{BourAlg8}. 
	The corresponding decomposition of unit into the sum of idempotents 
	can be lifted to $\End(M_{L})$ and we get a decomposition 
	$M_{L}=\bigoplus_{i=1}^{[L:K]}M_i$, where $k(M_i)\iso L$.
	This implies that $M_i$ are absolutely indecomposable.\\
4) If $L\sb K$ is a field of definition of $M$ then,
	for every $\si\in\Gal(K/L)$, we have $\si_*M\iso M$ and therefore
	$\si\in\Ga$. This implies that $K^\Ga\sb L$. 
	Let us show that $K^\Ga$ is a field of definition of $M$. 
  Let $L=K^\Ga$, $\si$ be some generator of $\Ga=\Gal(K/L)$
  and let $n=[K:L]$.
   To show that $M$ is defined over $L$, we have to define
   a $\Gal(K/L)$-equivariant structure on $M$. We will need
   to enlarge the field $K$ first.      
   By assumption, there exists an isomorphism $f:M\ar\si_*M$.
   We will write $f^r$ to denote the composition 
   $M\ar\si_*M\ar\si_*^2M\ar\dots\ar\si_*^rM$.
   Then $f^n:M\ar M$ is an automorphism of $M$.
   The group $\Aut(M)$ is finite
   and there exists some $n'\ge1$ such that $f^{nn'}=1$.
   Let $K'/K$ be a field extension of degree $n'$. 
   Let $\si'\in\Gal(K'/L)$ be some generator that is mapped
   to $\si\in\Gal(K/L)$. We can extend the map $f:M\ar\si_*M$
   to the map $f':M_{K'}\ar\si'_*M_{K'}$ by 
   $f'(m\ts a)=f(m)\ts\si'(a)$. 
   This map satisfies $(f')^{nn'}=1$
   and defines a structure of $\Gal(K'/L)$-equivariant sheaf
   on $M_{K'}$. 
   By Galois descent, there exists some $N\in\Coh X_L$ with
   $N_{K'}\iso M_{K'}$. It follows from the above lemma that 
   $N_{K}\iso M$.  \\
5) Let $A=\End(N)$ and let $\Ga=\Gal(K/k)$.
	It is known \cite[Proposition~8.4]{BourAlg8} 
	 that there exists an isomorphism of $k$-algebras
   $K\ts_kK\ar\prod_{\si\in\Ga}K_\si$, where $K_\si$ denotes a copy of $K$
   and the projection $K\ts_kK\ar K_\si$ is given by $x\ts y\mto x\si(y)$.
   By the Wedderburn-Malcev
   theorem \cite[Thoerem 6.2.1]{DrozdKir}, 
   there exists a section $K\ar A$ of the projection
   $A\ar K$. Applying $A\ts_K$ to the isomorphism
   $K\ts_kK\ar\prod_{\si\in\Ga}K_\si$   
   we get an isomorphism 
   $A\ts_kK\ar\prod_{\si\in\Ga}A_\si$, where $A_\si$ denotes a copy of $A$
   and the projection $A\ts_kK\ar A_\si$ is given by $x\ts y\mto x\si(y)$.
   Let $1=\sum_{\si\in\Ga}e_\si$ be the corresponding decomposition
   of the unit in $A\ts_kK$. Then we have $\tau(e_{\si})=e_{\si\tau\inv}$,
   where the action of \Ga on $A\ts_kK$ is given by the action
   on $K$. We get a decomposition 
   $N_K=\bigoplus_{\si\in\Ga}M_\si$, where the direct summand $M_\si$ 
   corresponds to the idempotent $e_\si$. 
   It follows from $k(N_K)\iso K\ts_kK\iso\prod_{\si\in\Ga}K_\si$
   that $k(M_\si)=K$ and the sheaves $M_\si$ are pairwise non-isomorphic.
   It follows from $\tau(e_{\si})=e_{\tau\inv\si}$ that 
   $\tau_*M_\si\iso M_{\tau\si}$.
   Let $M=M_1$. Then we have $N_K\iso\bigoplus_{\si\in\Gal(K/k)}\si_* M$.
   By the above lemma these $\Gal(K/k)$-equivariant sheaves
   descend to the same sheaf $N$ on $X$.
   It follows from $k(M)=K$ that $M$ is absolutely indecomposable.
   It follows from the condition $\si_*M\iso M_\si\not\iso M$ for
   every $\si\ne1$ that $M$ has minimal field of definition $K$.\\
6) Let $\bigoplus_{\si\in\Gal(K/k)}\si_* M$ descends to the sheaf $N$ over $X$.
	If $N'$ is a direct summand of $N$ than the set of isomorphism
	classes of indecomposable summands of $N'_K$ contains all
	$\si_*M$, $\si\in\Gal(K/k)$. They are pairwise non-isomorphic
	and therefore $N'_K=N_K$ and $N'=N$. This implies that $N$ is
	indecomposable. We have
	$$k(N_K)=k\big(\bigoplus_{\si\in\Gal(K/k)}\si_* M\big)=\prod_{\si\in\Gal(K/k)}k(M).$$
	This implies that $\dim_kk(N)=\dim_kk(M)$ and therefore $k(N)=k(M)$.
\end{proof}

\subsection{Stable bundles}\label{subsec:stable}
Let $X$ be a curve over a finite field $k$.
Define the slope of $\al=(n,d)\in\cN^*\xx\cZ$
to be $\mu(\al)=d/n$. For any vector bundle $M$ on $X$
(we identify vector bundles with locally free sheaves), define
its slope $\mu(M)=\mu(\ch M)$.

\begin{dfn}
We say that a bundle $M$ on $X$ is semistable (respectively, stable)
if, for any proper nonzero subbundle $N\sb M$, 
we have $\mu(N)\le\mu(M)$ (respectively, $\mu(N)<\mu(M)$). 
We say that a stable bundle $M$ is absolutely stable
if it stays stable after any finite field extension of $k$.
A semistable bundle is called polystable if it is a direct sum
of stable bundles.
For any $\al\in\cN^*\xx\cZ$, let $\lE_X(\al)$ be the set
of representatives of isomorphism classes of vector bundles
over $X$ having character \al.
Let $\lE_X\sst(\al)$ (respectively, $\lE_X\st(\al)$) be the subset of $\lE_X(\al)$ 
consisting of semistable (respectively, stable) bundles.
For any $\mu\in\cQ$, we define 
$$\lE_X\sst(\mu)=\bigcup_{\mu(\al)=\mu}\lE_X\sst(\al),\qquad
\lE_X\st(\mu)=\bigcup_{\mu(\al)=\mu}\lE_X\st(\al).$$
For any finite field extension $K/k$, we define
the sets $\lE_{X_{K}}(\al)$, $\lE_{X_{K}}\sst(\al)$, 
$\lE_{X_{K}}\st(\al)$, $\lE_{X_{K}}\sst(\mu)$,
$\lE_{X_{K}}\st(\mu)$
in the same way as above.
\end{dfn}

It is known that, for any $\al\in\cN^*\xx\cZ$, the set $\lE_X\sst(\al)$
is finite.
For any $\mu\in\cQ$, we define the category $\Coh_\mu X$
to be the subcategory of $\Coh X$ consisting of
semistable bundles having slope \mu. This is an abelian
category whose simple objects are precisely the stable bundles from 
$\lE_X\sst(\mu)$.
For any finite field extension $K/k$, define
$a_{\al}(K)$ to be the number of absolutely stable bundles in
$\lE\st_{X_{K}}(\al)$. For $r\ge1$, define 
$s_{\al,r}(K)$ to be the number of bundles in
$\lE\st_{X_K}(\al)$ with $\dim_K\End M=r$.
 
\begin{lmm}
\begin{sloppypar}
A stable bundle $M$ over $X$ is absolutely stable if and only if 
$\End M=~k$.
\end{sloppypar}
\end{lmm}
\begin{proof}
If $M\in\Coh X$ is stable then, for any finite extension
$K/k$, the sheaf $M_{K}$ is polystable. Indeed,
the relative socle (in the category $\Coh_\mu X$, $\mu=\mu(M)$)
of $M_{K}$ is $\Gal(K/k)$-stable and determines therefore
some $N\sb M$ with $\mu(N)=\mu$. It follows from stability of $M$ 
that $N=M$ and therefore the socle of $M_{K}$ is the whole
$M_{K}$. This implies that $M_K$ is polystable. If $\End M=k$
then $\End M_K=K$ and therefore $M_K$ is stable.
On the other hand, if $M$ is absolutely stable 
then it is absolutely indecomposable and we have $\End(M)=k(M)=k$.
\end{proof}

It follows that $a_\al(K)=s_{\al,1}(K)$ for any finite field extension $K/k$.
We will prove a more strong result now.

\begin{prp}\label{prp:stable corresp}
Let $K/k$ be a finite field extension and let $\Ga=\Gal(K/k)$.
For any absolutely stable sheaf $M$ over $X_K$ 
with the minimal field of definition $K$, the sheaf
$\bigoplus_{\si\in\Ga}\si_*M$ descends to a
stable sheaf over $X$ having
endomorphism ring isomorphic to $K$. The induced 
map between the corresponding sets is a quotient map 
with respect to the free action of $\Ga$ 
on the first set.
\end{prp}
\begin{proof}
All the sheaves $\si_*M$ are stable and therefore 
$\bigoplus_{\si\in\Ga}\si_*M$ 
is semistable, as also its descend $N$ over $X$.
Assume that there exists some nonzero $U\sb N$ with $\mu(U)=\mu(N)$.
Then $U$ is semistable as also $U_K\sb\bigoplus_{\si\in\Ga}\si_*M$. 
Any projection $U_K\ar\si_*M$ is either surjective or zero. 
One of them is surjective
and the inclusion $U_K\emb\bigoplus_{\si\in\Ga}\si_*M$ is $\Ga$-equivariant.
This implies that $U_K=\bigoplus_{\si\in\Ga}\si_*M$ and $U=N$. 
Thus, $N$ is stable. It follows from Proposition \ref{prp:indecomp}
that $\End(N)=k(N)=k(M)=K$.

Assume now that $N$ is some stable sheaf over $X$ with $\End(N)=K$. 
Then $N$ is indecomposable and its splitting field
equals $k(N)=\End(N)=K$. Applying Proposition \ref{prp:indecomp}
we can find an absolutely indecomposable sheaf $M$ with
a minimal field of definition $K$ such that 
$N_K\iso\bigoplus_{\si\in\Ga}\si_*M$.
As $N$ is stable, the sheaf $N_K$ is polystable. It follows that
$M$ is stable and therefore absolutely stable.
It is clear that if $M'$ is a different absolutely stable
sheaf over $X_K$ such that $N_K\iso\bigoplus_{\si\in\Ga}\si_*M'$
then $M'\iso\si_*M$ for some $\si\in\Ga$. This implies the last assertion
of the proposition.
\end{proof}

\begin{rmr}
If $\bigoplus_{\si\in\Ga}\si_*M$ descends to the sheaf $N$ then
$\ch N=[K:k]\ch M$. It follows from the proposition that $s_{\al,r}(k)\ne0$
implies $\al/r\in\cZ^2$.
\end{rmr}

\begin{lmm}\label{lmm:a and s}
Let $\cF_q/k$ be a finite field extension, $\al\in\cN^*\xx\cZ$ and $n\ge1$.
Then we have
$$a_{\al}(\cF_{q^n})=\sum_{r\mid n}rs_{r\al,r}(\cF_q).$$
\end{lmm}
\begin{proof}
Without loss of generality, we may assume that $k=\cF_q$.
Any absolutely stable sheaf over $\cF_{q^n}$ has the minimal
field of definition of the form $\cF_{q^r}$ with $r\mid n$.
It follows from Proposition \ref{prp:stable corresp} that the number
of absolutely stable sheaves over $\cF_{q^r}$ having character
$\al$ and minimal field of definition $\cF_{q^r}$
equals $rs_{r\al,r}(\cF_q)$. Now the statement of the lemma
is obvious.
\end{proof}

The goal of this paper is to determine the numbers $a_\al(\cF_{q})$
and the Poincar\'e polynomials of the corresponding moduli spaces. 

\subsection{Hall algebra}
Let $X$ be a smooth projective curve over a finite field $k=\cF_q$.
Let \lH be the Hall algebra \cite{Ringel1} of the category $\Coh X$ and let $\lH_+$
be its subalgebra generated by the sheaves having nonnegative degree.
We endow $\lH_+$ with an $\cN$-grading, where an isomorphism class
$[M]$ has degree $\rk M+\deg M$. 
Let $\what\lH_+$ be the corresponding completion. 
Let us endow the algebra $\cQ[x_1,x_2]$ with a new product
given by the formula
$$x^\al\circ x^\be=q^{-\ang{\al,\be}}x^{\al+\be}.$$
This new algebra is denoted by $\cQ\tw_q[x_1,x_2]$ or 
$\cQ\tw_{\cF_q}[x_1,x_2]$.

\begin{lmm}[see {\cite[Lemma 3.3]{Rei1}}]\label{lmm:integral}
The map $\int:\lH_+\ar\cQ\tw_q[x_1,x_2]$ given by
$$[M]\mto \frac1{\#\Aut M}x^{\ch M}$$
is a homomorphism of algebras. It induces a homomorphism
of completions \br $\int:\what\lH_+\ar\cQ\tw_q\pser{x_1,x_2}$.
\end{lmm}

\section{Solution of the HN-recursion}\label{sec:HN rec}
Let \Ga be a commutative semigroup
equipped with a total preorder $\le$ on \Ga (e.g. a preorder given by
some stability on $\cN^I\ms\set0$, see \cite{MR1}) such that
$$\min\set{\al,\be}\le\al+\be\le\max\set{\al,\be},\qquad \al,\be\in\Ga.$$
We will write $\al<\be$ if $\al\le\be$ but $\be\not\le\al$.

\begin{dfn}
For any sequence $\la=(\la_1,\dots,\la_k)$ of elements in \Ga, 
we define $|\la|=\sum\la_i$ and $l(\la)=k$.
We will say that a sequence $\la=(\la_1,\dots,\la_k)$ is stable if
$\sum_{i=1}^j\la_i<|\la|$ for $1\le j<k$.
For any $\al\in\Ga$, let $\lP_\al$ be the set of sequences $\la$ in \Ga
such that $|\la|=\al$. Let $\lP_\al\st\sb\lP_\al$ be a subset
consisting of stable sequences and let $\lP^i_\al\sb\lP_\al$ be a subset
consisting of strictly increasing sequences.
\end{dfn}
 
The following result is an easy generalization of \cite[Theorem 5.1]{Rei2}.

\begin{thr}
Assume that, for any $\al\in\Ga$, the set $\lP_\al\st$ is finite.
Let $A$ be an associative ring and let $(a_\al)_{\al\in\Ga}$,
$(b_\al)_{\al\in\Ga}$ be two systems of elements in $A$ 
related by the formula
\eq\label{eq:recurs}
b_\al=\sum_{\la\in\lP_\al^i}a_{\la_1}\dots a_{\la_{l(\la)}}.
\endeq
Then we have
\eq\label{eq:rec solution}
a_\al=\sum_{\la\in\lP_\al\st}(-1)^{k-1}b_{\la_1}\dots b_{\la_{l(\la)}}.
\endeq
\end{thr}
\begin{proof}
We use the induction on the number of elements in $\lP\st_\al$.
Given a sequence $\la=(\la_1,\dots,\la_k)$, we say that
a set of positive integers 
$$S=(s_1<\dots<s_{r-1}<s_r=k)$$ 
is $\la$-admissible
if the sequences $\mu^i=(\la_{s_{i-1}+1},\dots,\la_{s_i})$, $i=1,\dots,r$ 
(we set $s_0=0$) are stable and $|\mu^1|<\dots<|\mu^r|$.
If there exists any \la-admissible set, then \la is stable. 
Let $\lA(\la)$ be the set of all $\la$-admissible sets.
For any sequence $\la=(\la_1,\dots,\la_k)$,
we define $a_\la=a_{\la_1}\dots a_{\la_k}$ and 
$b_\la=b_{\la_1}\dots b_{\la_k}$.
Then
\begin{align*}
a_\al+\sum_{\la\in\lP^s_\al}(-1)^{l(\la)}b_\la
&=a_\al-b_\al+\sum_{\la\in\lP^s_\al,l(\la)>1}(-1)^{l(\la)}b_\la\\
&=-\sum_{\mu\in\lP^i_\al,l(\mu)>1}a_\mu
	+\sum_{\la\in\lP^s_\al,l(\la)>1}(-1)^{l(\la)}b_\la\\
&=-\sum_{\la\in\lP_\al}\sum_{\stackrel{S\in\lA(\la)}{\#S>1}}(-1)^{l(\la)-\#S}b_\la
	+\sum_{\la\in\lP^s_\al,l(\la)>1}(-1)^{l(\la)}b_\la\\
&=\sum_{\la\in\lP_\al,l(\la)>1}(-1)^{l(\la)-1}b_\la
  \sum_{S\in\lA(\la)}(-1)^{\#S}.
\end{align*}
It is proved in \cite[Lemma 5.4]{Rei2} that for any sequence \la of length $l(\la)>1$
$$\sum_{S\in \lA(\la)}(-1)^{\#S}=0.$$
This implies the theorem.
\end{proof}

The case that we are actually interested in in this paper is 
$\Ga=\cN^*\xx\cZ$ with the preorder given by $\al\le\be$ if $\mu(\al)\le\mu(\be)$.
Unfortunately, the above theorem can not be applied in this
situation, as $\lP_\al\st$ (and also $\lP_\al^i$) 
are not finite sets. Even in the complete topological
rings, the convergence of the sums \eqref{eq:recurs} do not
imply the convergence of the sums \eqref{eq:rec solution}.
Still, under certain conditions,
there exists a nice solution of the recursive formula due to Zagier.
Given a topological space $A$
and a system of elements $(f_\la)_{\la\in\lP_\al^i}$,
we say that the sum 
$\sum_{\la\in\lP_\al^i}f_\la$ converges if there exists a limit 
$$\lim_{n\to\infty}\sum_{\stackrel{\la\in\lP_\al^i}{\mu(\la_{l(\la)})<n}}f_\la.$$

%Instead of summing over all stable partitions
%of $\al=(n,d)\in\Ga$ in the formula \eqref{eq:recurs}, we can 
%sum over the tuples $n_*=(n_1,\dots,n_k)$ 
%of positive integers such that $\sum n_i=n$. Then, for any such tuple,
%we sum over the set $D(n_*,d)$ of all $k$-tuples of integers
%$(d_1,\dots,d_k)$ such that $\sum d_i=d$ and 
%$((n_1,d_1),\dots,(n_k,d_k))$ is stable.

\begin{thr}[Zagier, \protect{\cite[Theorem 3]{Zagier1}}]\label{thr:zagier formula}
Let $A$ be a complete normed ring and let $t\in A$ be an invertible 
element with $|t|<1$. Let $(a_\al)_{\al\in\Ga}$ and
$(b_n)_{n\ge1}$ be two systems of elements in $A$ 
related by
$$b_n=\sum_{\la\in\lP_\al^i}t^{\sum_{i<j}\ang{\la_i,\la_j}}
a_{\la_1}\dots a_{\la_{l(\la)}},\qquad\al=(n,d)\in\Ga,$$
where the pairing $\ang{\cdot\,,\cdot}$ on $\cZ^2$ is given by
$\ang{(n,d),(n',d')}=nd'-dn'$. Then
$$a_\al=\sum_{\stackrel{n_1,\dots,n_k>0}{n_1+\dots+n_k=n}}
\Phi_{n_*,d}(t)b_{n_1}\dots b_{n_k},\qquad \al=(n,d)\in\Ga,$$
where $n_*=(n_1,\dots,n_k)$,
$$\Phi_{n_*,d}(t)=(-1)^{k-1}\prod_{i=1}^{k-1}\frac
{t^{(n_i+n_{i+1})\set{(n_1+\dots+n_i)d/n}_+}}
{1-t^{n_i+n_{i+1}}},$$
and $\set{x}_+=\max\set{x-\lfloor x\rfloor,1}$.
\end{thr}

\begin{rmr}\label{rmr:zagier formula}
Assume that $A$ is a commutative complete normed ring,
$t\in A$ is invertible with $|t|>1$, 
and $(a_\al)_{\al\in\Ga}$,
$(b_n)_{n\ge1}$ are two systems of elements in $A$
satisfying
$$b_n=\sum_{\la\in\lP_\al^i}t^{-\sum_{i<j}\ang{\la_i,\la_j}}
a_{\la_1}\dots a_{\la_{l(\la)}},\quad\al=(n,d)\in\Ga.$$
Then
$$a_\al=\sum_{\stackrel{n_1,\dots,n_k>0}{n_1+\dots+n_k=n}}
\Psi_{n_*,d}(t)b_{n_1}\dots b_{n_k},\quad \al=(n,d)\in\Ga,$$
where $n_*=(n_1,\dots,n_k)$, 
$$\Psi_{n_*,d}(t)=\prod_{i=1}^{k-1}\frac
{t^{(n_i+n_{i+1})\set{(n_1+\dots+n_i)d/n}}}
{1-t^{n_i+n_{i+1}}},$$
\begin{sloppypar}\noindent
and $\set{x}=x-\lfloor x\rfloor$.
This follows from the above theorem and the equality 
$\Phi_{n_*,d}(t\inv)=\Psi_{\ub n_*,d}(t)$, where 
$\ub n_*=(n_k,\dots,n_1)$.
\end{sloppypar}
\end{rmr}

%\begin{thr}[Zagier \cite{}]
%We have
%$$Q_{n_*,d}(t)=\sum_{d_*\in D(n_*,d)}t^{\sum_{i<j}(d_jn_i-d_in_j)}
%=\prod_{i=1}^{k-1}\frac
%{t^{(n_i+n_{i+1})\set{(n_1+\dots+n_i)d/n}_+}}
%{1-t^{n_i+n_{i+1}}}$$
%in $\cQ\lser t$, where $\set{a}_+=\max\set{a-\lfloor a\rfloor,1}$.
%\end{thr}
%
%
%\begin{rmr}
%For $n_*=(n_1,\dots,n_k)$, we have
%$$Q_{n_*,d}(t\inv)=(-1)^{k-1}\prod_{i=1}^{k-1}\frac
%{t^{(n_i+n_{i+1})\set{(n_1+\dots+n_i)d/n}}}
%{1-t^{n_i+n_{i+1}}}.$$
%\end{rmr}

 %recursive formula
\begin{sloppypar}
	
\end{sloppypar}
\section{\texorpdfstring{\la}{Lambda}-ring of c-sequences}\label{sec:c-seq}
Let $S$ be the set of sequences $s=(s_k)_{k\ge1}$ of complex
numbers. For any two sequences $s=(s_k)_{k\ge1}$, $s'=(s'_k)_{k\ge1}$, 
define their sum and product by
$$s+s'=(s_k+s'_k)_{k\ge1},\quad ss'=(s_ks'_k)_{k\ge1}.$$
This endows $S$ with a structure of a ring.
Define the Adams operations on $S$ by $\psi_r(s)=(s_{rk})_{k\ge1}$, $r\ge1$.
This endows $S$ with a structure of a \la-ring.

\begin{rmr}
For any commutative \cQ-algebra $R$, we can define the \la-ring
$S(R)$ of sequences in $R$ in the same way as above. There is
a commutative diagram of \la-ring isomorphisms
\begin{diagram}
W(R)&&\rTo^a&&S(R)\\
&\rdTo_b&&\ldTo_c\\
&&\La(R)\\
\end{diagram}
where $W(R)$ is the ring of Witt vectors over $R$, $\La(R)=1+tR\pser t$ is 
Grothendieck's \la-ring (see, e.g., \cite[V.2.3]{SGA6}), 
and the maps $a$, $b$, and $c$ are given by
$$a:(x_n)_{n\ge1}\mto\Big(\sum_{d\mid n}dx_d^{n/d}\Big)_{n\ge1},$$
$$b:(x_n)_{n\ge1}\mto\prod_{n\ge1}(1-x_nt^n)\inv,$$
$$c:(s_n)_{n\ge1}\mto\exp\Big(\sum_{n\ge1}s_nt^n/n\Big).$$  
\end{rmr}

Let $\cF_q$ be some finite field with $q$ elements.
Define the subgroups of $\cC^*$
$$G_0=\sets{\la\in\cC^*}{|\la|=q^{n/2},n\in\cZ},\qquad G_1=\sets{q^{n/2}}{n\in\cZ}.$$
We endow the group algebra $\cQ[G_0]$ with a structure of a \la-ring by the formula
$\psi_r(\sum_{\la\in G_0}a_\la\la)=\sum_{\la\in G_0}a_\la\la^r$.
Consider the map $\cQ[G_0]\ar S$, given by
$$\cQ[G_0]\ni\sum_{\la\in G_0}a_\la\la\mto\Big(\sum_{\la\in G_0}a_\la\la^k\Big)_{k\ge1}\in S.$$

\begin{prp}
The map $\cQ[G_0]\ar S$ is an injective \la-ring homomorphism.
If $s\in S$ is the image of $\sum a_\la\la\in\cZ[G_0]$ then
$$Z_s(t):=\exp\Big(\sum_{k\ge1}s_k\frac{t^k}k\Big)=\prod_{\la\in G_0}(1-t\la)^{-a_\la}$$
in $\cC\pser t$.
\end{prp}
\begin{proof}
It is clear that $\cQ[G_0]\ar S$ is a \la-ring homomorphism.
To prove the formula, we note that
$$\log\Big(\prod_{\la\in G_0}(1-t\la)^{-a_\la}\Big)
=\sum_{\la\in G_0}a_\la\log\Big(\frac1{1-t\la}\Big)
=\sum_{k\ge1}\sum_{\la\in G_0}a_\la\frac{(t\la)^k}k
=\sum_{k\ge1}\frac{s_kt^k}k.$$

If the map $\cQ[G_0]\ar S$ is not injective, then we can find some
nonzero element $\sum a_\la\la\in\cZ[G_0]$ in the kernel.
It follows that $\prod_{\la\in G_0} (1-t\la)^{-a_\la}=1$ and therefore all $a_\la$
are zeros, contradicting the assumption.
\end{proof}

\begin{dfn}
Define the ring $S_c^0$ of effective c-sequences to be the image
of $\cQ[G_0]$ in $S$. Define the ring of special c-sequences to be
the image of $\cQ[G_1]$ in $S$.
Define the ring $S_c$ of c-sequences to be localization  
of $S_c^0$ with respect to special c-sequences that are invertible in
$S$. The ring $S_c$ is a \la-subring of $S$.
\end{dfn}

\begin{rmr}
A sequence $s\in S$ is an effective c-sequence 
if and only if there exists some integer $r\ge1$ such
that $Z_s(t)^r=\exp(\sum_{k\ge1}s_kt^k/k)^r$ is a rational function
with zeros and poles having absolute values of the form $q^{n/2}$,
$n\in\cZ$.
\end{rmr}

\begin{rmr}\label{rmr:c-seq of var}
If $X$ is a scheme over $\cF_q$ then, by the Weil conjectures,
the sequence $(\#X(\cF_{q^k}))_{k\ge1}$ is an effective
c-sequence.
\end{rmr}

\begin{rmr}
\begin{sloppypar}	
The special c-sequence $\cL=(q^k)_{k\ge1}$ is called 
the Lefschetz c-sequence. This is a c-sequence associated
to the affine line $\cA^1_{\cF_q}$. 
The ring of special c-sequences is generated by $\cL$ and $\cL\inv$.
\end{sloppypar}
\end{rmr}

\begin{rmr}
There exists a structure of a \la-ring on the Grothendieck ring 
of algebraic varieties over \cC 
(or, more precisely, on the Grothendieck ring of the category
of motives over \cC), see \cite{Getz1}. 
In this structure the sigma-operations on algebraic varieties
are given by the symmetric products.
In the case of a finite field $\cF_q$, we define
the \la-ring structure not on the Grothendieck ring 
of algebraic varieties over $\cF_q$ or $\ub\cF_q$ ,
but rather on the ring of c-sequences of algebraic varieties. 
The Adams operation $\psi_n$ applied to the c-sequence of the algebraic variety
$X$ over $\cF_q$ gives the c-sequence of the algebraic variety
$X_{\cF_{q^n}}$ over $\cF_{q^n}$.  
\end{rmr}

\begin{dfn}
For any $\la\in G_0$, define its weight $w(\la)$ to be the integer
such that $|\la|=q^{w(\la)/2}$. This defines a group homomorphism
$w:G_0\ar\cZ$, which induces a \la-ring homomorphism 
$P:S_c^0\iso\cQ[G_0]\ar\cQ[v,v\inv]$, $\sum a_\la\la\mto\sum a_\la v^{w(\la)}$,
called the Poincar\'e polynomial. It induces the map
$P:S_c\ar\cQ(v)$, called the Poincar\'e function.  
\end{dfn}

\begin{rmr}
If $s\in S_c^0$ is an effective c-sequence and $r\ge1$ is an
integer such that $Z_s(t)^r$ is a rational function, say
$\prod_{\la\in G_0}(1-\la t)^{-a_\la}$, then the Poincar'e 
polynomial of $s$ equals $P(s,v)=\frac 1r\sum_{\la\in G_0} a_\la v^{w(\la)}$.
\end{rmr}

\begin{rmr}\label{rmr:poincare of variety}
Let $X$ be a scheme over $\cF_q$ and let $l$ be a prime
number, coprime with~$q$. There is a weight filtration
$(W_n)_{n\in\cZ}$ defined
on the cohomology groups $H_c^i(X,\cQ_l)$
\cite{Del2}. One defines the virtual Betti numbers
$$b_n(X)=\sum_{i\ge0}(-1)^i\dim(\Gr^W_nH^i_c(X,\cQ_l))$$
and the virtual Poincar\'e polynomial 
$P(X,v)=\sum_{n\in\cZ}b_n(X)v^n$.
If $s\in S_c^0$ is a c-sequence associated to $X$ 
as in Remark \ref{rmr:c-seq of var} then
$P(X,v)=P(s,v)$.
\end{rmr}

\section{c-sequences of semistable bundles}\label{sec:semist}
Let $X$ be a curve of genus $g$ over a finite field $k=\cF_q$. We
introduce the pairing $\ang{\cdot,\cdot}$ on $\cZ^2$ as in Section
\ref{subsec:chern}. Let $\Ga=\cN^*\xx\cZ$ be endowed with the
total preorder by the rule $\al\le\be$ if $\mu(\al)\le\mu(\be)$.
For any $\al\in\Ga$ and any finite field extension $K/k$, we
define
$$r_{\al}(K)=\sum_{M\in\lE\sst_{X_{K}}(\al)}\frac1{\#\Aut M},\quad
m_{\al}(K)=\sum_{M\in\lE_{X_{K}}(\al)}\frac1{\#\Aut M}.$$
The first sum is finite and the second sum converges to a rational number
(see Theorem \ref{thr:HNDR}).
Define the elements $r_\al,m_\al\in S$ by $r_\al=(r_{\al}(\cF_{q^k}))_{k\ge1}$
and $m_\al=(m_{\al}(\cF_{q^k}))_{k\ge1}$.
The goal of this section is to show that $r_\al$ is a c-sequence and
to introduce the formula for its Poincar\'e function due to Zagier.

The idea goes as follows. Let us imagine that  the elements
$$M_\al=\sum_{M\in\lE_X(\al)}[M],\qquad R_\al=\sum_{M\in\lE_X\sst(\al)}[M]$$
are well defined in some completion of the Hall algebra of $\Coh X$.
Using the Harder-Narasimhan filtration, we can write
$$M_\al=\sum_{\la\in\lP^i_\al}R_{\la_1}\dots R_{\la_{l(\la)}}.$$
If we could apply the map
$\int:\lH_+\ar\cQ_q\tw[x_1,x_2]$ to the above formula then we would get
the relation between the numbers $m_\al(k)$ and
$r_\al(k)$.

\begin{thr}[see \protect{\cite{HarNar1,DesRam1}}]\label{thr:HNDR}
For any $\al=(n,d)\in\Ga$, we have
$$m_\al(\cF_q)=\sum_{\la\in\lP^i_\al}q^{-\sum_{i<j}\ang{\la_i,\la_j}}
r_{\la_1}(\cF_q)\dots r_{\la_{l(\la)}}(\cF_q).$$
Moreover,
$$m_\al(\cF_q)
=m_n(\cF_q)=\frac{\prod_{i=1}^{2g}(1-\om_i)}{q-1}
q^{(n^2-1)(g-1)}Z_X(q^{-2})\dots Z_X(q^{-n}),$$
where
$$Z_X(t)=\frac{\prod_{i=1}^{2g}(1-\om_it)}{(1-t)(1-qt)}$$
is a zeta-function of $X$ and $|\om_i|=q^{1/2}$.
\end{thr}

\begin{crl}
For any $n\ge1$, $m_n=(m_n(\cF_{q^k}))_{k\ge1}$ is a c-sequence with the Poincar\'e function
$$P(m_n,v)=(v^{2n}-1)\prod_{i=1}^{n}\frac{(1-v^{2i-1})^{2g}}{(1-v^{2i})^2}.$$
\end{crl}

%Applying \ref{}, we get
%$$r_\al(\cF_q)=\sum_{\la\in\lP^s_\al}(-1)^{l(\la)-1}q^{-\sum_{i<j}\ang{\la_i,\la_j}}
%m_{\la_1}(\cF_q)\dots m_{\la_{l(\la)}}(\cF_q)$$
%and therefore
%$$r_\al=\sum_{\la\in\lP^s_\al}(-1)^{l(\la)-1}\cL^{-\sum_{i<j}\ang{\la_i,\la_j}}
%m_{\la_1}\dots m_{\la_{l(\la)}}$$
%in $S$.

The sum in the above theorem is infinite, so we can not deduce
directly that $r_\al$ is a c-sequence.

\begin{thr}
For any $\al=(n,d)$, we have
$$r_\al(\cF_q)
=\sum_{\stackrel{n_1,\dots,n_k>0}{\sum n_i=n}}
q^{(g-1)\sum_{i<j}n_in_j}\Psi_{n_*,d}(q)m_{n_1}(\cF_q)\dots m_{n_k}(\cF_q)$$
where
$$\Psi_{n_*,d}(t)=\prod_{i=1}^{k-1}\frac
{t^{(n_i+n_{i+1})\set{(n_1+\dots+n_i)d/n}}}
{1-t^{n_i+n_{i+1}}}.$$
\end{thr}
\begin{proof}
Let $a_\al=q^{n^2(1-g)/2}r_\al(\cF_q)$, $\al=(n,d)\in\Ga$
and let $b_n=q^{n^2(1-g)/2}m_n(\cF_q)$, $n\ge1$.
Then the above theorem implies
$$b_n=\sum_{\la\in\lP^i_\al}q^{-\sum_{i<j}\ang{\la_i,\la_j}'}
a_{\la_1}\dots a_{\la_{l(\la)}},\qquad \al=(n,d)\in\Ga,$$
where $\ang{(n,d),(n',d')}'=nd'-dn'$.
Applying Remark \ref{rmr:zagier formula}, we get
$$a_\al=\sum_{\stackrel{n_1,\dots,n_k>0}{\sum n_i=n}}
\Psi_{n_*,d}(q)b_{n_1}\dots b_{n_k},\qquad \al=(n,d)\in\Ga,$$
which implies the theorem.
\end{proof}

\begin{crl}
For any $\al=(n,d)$, we have
$$r_\al
=\sum_{\stackrel{n_1,\dots,n_k>0}{\sum n_i=n}}
\cL^{(g-1)\sum_{i<j}n_in_j}\Psi_{n_*,d}(\cL)m_{n_1}\dots m_{n_k}.$$
\end{crl}

\begin{crl}\label{cor:poincare of r}
For any $\al=(n,d)\in\Ga$, the element $r_\al\in S$
is a c-sequence and its Poincar\'e function equals
\begin{align*}
P(r_\al,v)
&=\sum_{\stackrel{n_1,\dots,n_k>0}{\sum n_i=n}}
v^{2(g-1)\sum_{i<j}n_in_j}\Psi_{n_*,d}(v^2)P_{n_1}(v)\dots P_{n_k}(v),
\end{align*}
where
$$P_n(v)=(v^{2n}-1)\prod_{i=1}^{n}\frac{(1-v^{2i-1})^{2g}}{(1-v^{2i})^2}.$$
\end{crl}

%\begin{proof}
%It follows from Zagier formula that
%\begin{align*}
%r_\al
%&=\sum_{\la\in\lP^s_\al}(-1)^{l(\la)-1}\cL^{-\sum_{i<j}\ang{\la_i,\la_j}}
%   m_{\la_1}\dots m_{\la_{l(\la)}}\\
%&=\sum_{\stackrel{n_1,\dots,n_k>0}{\sum n_i=n}}
%  (-1)^{k-1}\cL^{(g-1)\sum_{i<j}n_in_j}m_{n_1}\dots m_{n_k}
%  \sum_{d_*\in D(n_*,d)}\cL^{\sum_{i<j}(d_in_j-d_jn_i)}\\
%&=\sum_{\stackrel{n_1,\dots,n_k>0}{\sum n_i=n}}
%  (-1)^{k-1}\cL^{(g-1)\sum_{i<j}n_in_j}Q_{n_*,d}(\cL^{-1})m_{n_1}\dots m_{n_k}.
%\end{align*}
%This implies that $r_\al$ is a c-sequence.
%It follows now, from $P(\cL,v)=v^2$, the second expression for $P(r_\al,v)$.
%Applying to it the Zagier formula, we get the first expression for
%$P(r_\al,v)$.
%\end{proof}
 %r_\alpha
\section{Number of stable bundles}\label{sec:stable}
In the last section we have seen how to compute the Poincar\'e 
function of the c-sequence $r_\al$. In this section we will 
prove a formula relating the Poincar\'e functions of $a_\al$ 
and $r_\al$. Without loss of generality, we may assume that
\al has nonnegative second coordinate or, equivalently, nonnegative
slope $\mu=\mu(\al)$. 

Recall, that with any curve $X$ over a finite field $k=\cF_q$, we have associated
the numbers $a_\al(K)$ and $s_{\al,r}(K)$, where $K/k$ is a finite
field extension and $r\ge1$. 
The totality of elements $a_\al(K)$ (respectively, $s_{\al,r}(K)$) 
for finite field extensions $K/k$
defines the element $a_\al$ (respectively, $s_{\al,r}$) in $S$. 

\begin{lmm}
We have
$$\psi_n(a_{\al})=\sum_{k\mid n}ks_{k\al,k}$$
in the \la-ring $S$.
\end{lmm}
\begin{proof}
Let $k=\cF_q$.
The $r$-th component of $\psi_n(a_{\al})$ equals $a_\al(\cF_{q^{nr}})$
and by Lemma \ref{lmm:a and s} we have
$$a_\al(\cF_{q^{nr}})=\sum_{k\mid n}ks_{k\al,k}(\cF_{q^r}).$$
\end{proof}

We can apply Lemma \ref{lmm:power} to the result of the above lemma in the \la-ring $S\pser x$.

\begin{crl}\label{crl:pow a_al}
For any element $f\in 1+xS\pser x$, we have
$$\Pow(f,a_\al)=\prod_{k\ge1}\psi_k(f)^{s_{k\al,k}}.$$
\end{crl}

\begin{rmr}\label{rmr:poincare of moduli}
For any $\al\in\cN^*\xx\cZ$, one can construct
the moduli space $\lM(\al)$ of stable bundles over $X_{\ub k}$
having character \al, see \cite{HarNar1}.
This moduli space is defined over some finite field extension of $k$
(let us assume that over $k$ itself). 
The number of $\cF_{q^n}$-rational points of $\lM(\al)$
equals $a_\al(\cF_{q^n})$.
This implies that the virtual Poincar\'e polynomial of $\lM(\al)$ 
equals the Poincar\'e polynomial of the c-sequence $a_\al$. 
If $X$ is a curve over complex numbers then we can find its form
over some ring finitely generated over \cZ. This form can be reduced
to finite fields $\cF_p$ so that, for almost all $p$, the reduction 
is again a smooth projective curve. The virtual Poincar\'e
polynomials of the moduli spaces of stable vector bundles over these
curves coincide with the virtual Poincar\'e of the moduli space
of stable bundles on a complex curve. This means that the last invariants
are given by the Poincar\'e functions of the c-sequences
$a_\al$. In this section we prove a formula for these c-sequences
and their Poincar\'e polynomials.
\end{rmr}

For any $\al\in\cN^*\xx\cZ$, there are only finitely many 
isomorphism classes of semistable bundles on $X$ having character $\al$. 
Define the elements
$$R_\al(K)=\sum_{[M]\in\lE\sst_{X_K}(\al)}[M]$$
and
$$R_\mu(K)=\sum_{M\in\lE\sst_{X_K}(\mu)}[M]=[0]+\sum_{\mu(\al)=\mu}R_\al(K)$$
in $\what\lH_+(X_K)$. 
The elements $R_\mu(K)$ are invertible 
in $\what\lH_+(X_K)$.

\begin{lmm}(see {\cite[Lemma 3.4]{Rei1}})\label{lmm:inverse R}
We have
$$R_\mu(K)\inv=\sum_{M\in\lE\sst_{X_K}(\mu)}\ga_M[M],$$
where
$$\ga_M=
\begin{cases}
0&
\text{if }M\text{ is not polystable},\\
\prod_{S\in\lE\st_{X_K}(\mu)}(-1)^{m_S}(\#\End S)^{\binom{m_S}2}&
\text{if }M=\bigoplus_{S\in\lE\st_{X_K}(\mu)} S^{m_S}.
\end{cases}
$$
\end{lmm}

Define the generating functions
$$r_\mu(K)=1+\sum_{\mu(\al)=\mu}r_\al(K)x^\al,\quad 
a_\mu(K)=\sum_{\mu(\al)=\mu}a_\al(K) x^\al$$
in $\cQ\pser{x_1,x_2}$.
It is clear that $r_\mu(K)=\int R_\mu(K)$.
The totality of the elements $r_\mu(K)$ and $a_\mu(K)$
for finite field extensions $K/\cF_q$ defines the elements
$r_\mu,a_\mu\in S\pser{x_1,x_2}$. We can write also
$$r_\mu=1+\sum_{\mu(\al)=\mu}r_\al x^\al,\quad 
a_\mu=\sum_{\mu(\al)=\mu}a_\al x^\al.$$

Recall, that the element $\cL\in S$ is defined by $\cL=(q^k)_{k\ge1}$. This
is a c-sequence with a Poincar\'e polynomial $P(\cL,v)=v^2$. 
Define a new multiplication on $S\pser{x_1,x_2}$ 
by the formula
$$x^\al\circ x^\be=\cL^{-\ang{\al,\be}}x^{\al+\be}.$$
The new algebra is denoted by $S\tw_\cL\pser{x_1,x_2}$.

\begin{thr}
We have
$$r_\mu\circ\Exp\left(\frac{a_\mu}{1-\cL}\right)=1$$ 
in $S\tw_\cL\pser{x_1,x_2}$.
\end{thr}
\begin{proof}
Consider some finite field extension $K$ of $\cF_q$. The corresponding component
of $r_\mu$ is given by $r_\mu(K)=\int R_\mu(K)$.
Its inverse in $\cQ\tw_K\pser{x_1,x_2}$ is given,
in view of Lemma \ref{lmm:inverse R} and Lemma \ref{lmm:integral},
by the formula

\begin{align*}
\int R_\mu(K)\inv
&=\sum_{(m_S)_{S\in\lE\st}}\frac{\prod_{S\in\lE\st}(-1)^{m_S}(\#\End S)^{\binom{m_S}2}}
{\#\Aut(\oplus_{S\in\lE\st} S^{m_S})}x^{\sum_{S\in\lE\st} m_S\ch S}\\
&=\sum_{(m_S)_{S\in\lE\st}}\prod_{S\in\lE\st}
\left(\frac{(-1)^{m_S}(\#\End S)^{\binom{m_S}2}}{\#\GL_{m_S}(\End S)}x^{m_S\ch S}\right)\\
&=\prod_{S\in\lE\st}
\left(\sum_{m\ge0}\frac{(-1)^{m}(\#\End S)^{\binom{m}2}}{\#\GL_{m}(\End S)}x^{m\ch S}\right),
\end{align*}
where by $\lE\st$ we denote the set $\lE\st_{X_K}(\mu)$.
For any stable bundle $S$, the ring of endomorphisms $\End S$ is a finite field,
say $\cF_s$. We have
$$\frac{(-1)^m s^{\binom m2}}{\#\GL_m(\cF_s)}
=\prod_{i=1}^m(1-s^i)\inv=[\infty,m]_s.$$
This implies
\begin{multline*}
\int R_\mu(K)\inv
=\prod_{S\in\lE\st}\Big(\sum_{m\ge0}[\infty,m]_vx^{m\ch S}\big|_{v=\#\End S}\Big)\\
=\prod_{\stackrel{\mu(\al)=\mu}{r\ge1}}\Big(\sum_{m\ge0}\psi_r([\infty,m]_v)x^{m\al}\big|_{v=\#K}\Big)
^{s_{\al,r}(K)}.
\end{multline*}
It follows that we have in $S\tw_\cL\pser{x_1,x_2}$ 
\begin{multline*}
r_\mu\inv
=\prod_{\al,r}\Big(\sum_{m\ge0}\psi_r([\infty,m]_v)x^{m\al}\big|_{v=\cL}\Big)^{s_{\al,r}}
=\prod_{\al,r}\psi_r\Big(\sum_{m\ge0}[\infty,m]_vx^{m\al/r}\big|_{v=\cL}\Big)^{s_{\al,r}}\\
=\prod_{\al,r}\psi_r\Big(\sum_{m\ge0}[\infty,m]_vx^{m\al}\big|_{v=\cL}\Big)^{s_{r\al,r}}
\rEq^{\text{Cor. }\ref{crl:pow a_al}}
\prod_\al\Pow\Big(\sum_{m\ge0}[\infty,m]_vx^{m\al}\big|_{v=\cL},a_\al\Big)\\
=\prod_\al\Exp\Big(a_\al\Log\Big(\sum_{m\ge0}[\infty,m]_vx^{m\al}\big|_{v=\cL}\Big)\Big)
=\prod_\al\Exp\Big(a_\al\frac{x^\al}{1-\cL}\Big)
=\Exp\Big(\frac{a_\mu}{1-\cL}\Big),
\end{multline*}
where we have used the formula
$$\Exp\left(\frac x{1-\cL}\right)=\sum_{m\ge0}[\infty,m]_vx^m\big|_{v=\cL}$$
in $S\pser x$. It is obtained from Lemma \ref{lmm:heine} as follows. Let $R$ be a localization
of $\cQ[v]$ with respect to $v$ and $(1-v^k)$, $k\ge1$. It is a 
\la-subring of $\cQ(v)$ and all the components of the Heine formula
are contained in $R$. Now we apply the \la-ring homomorphism
$R\ar S$, $v\mto\cL$.
\end{proof}

Let us endow the algebra $\cQ(v)\pser{x_1,x_2}$ with a new
multiplication
$$x^\al\circ x^\be=v^{-2\ang{\al,\be}}x^{\al+\be}.$$
The new algebra is denoted by $\cQ(v)\tw_{v^2}\pser{x_1,x_2}$.

\begin{crl}
We have
$$\Big(1+\sum_{\mu(\al)=\mu}P(r_\al,v)x^\al\Big)
\circ\Exp\left(\frac{\sum_{\mu(\al)=\mu}P(a_\al,v)x^\al}{1-v^2}\right)=1$$ 
in $\cQ(v)\tw_{v^2}\pser{x_1,x_2}$.
\end{crl}

This formula allows us to compute the Poincar\'e polynomial of $a_\al$
using the formula for the Poincar\'e function of $r_\al$ (see Corollary
\ref{cor:poincare of r}). By Remark \ref{rmr:poincare of moduli}
this gives us the virtual Poincar\'e polynomials 
$$P(\lM(\al),v)=P(a_\al,v)$$
of the moduli spaces of stable sheaves on a complex curve.

\begin{rmr}
For any $\al\in\Ga$, let $\ub\lM(\al)$ denote the moduli space
of semistable sheaves on $X$ having character \al. Then we have
$$1+\sum_{\mu(\al)=\mu}P(\ub\lM(\al),v)x^\al
=\Exp\Big(\sum_{\mu(\al)=\mu}P(\lM(\al),v)x^\al\Big).$$ 
\end{rmr}

\begin{rmr}
Note that the series $a_\mu$ and $r_\mu$ and their Poincar\'e 
functions can be considered as series in one variable.
Indeed, there exists a unique $\ga\in\cN^*\xx\cZ$ having coprime components
and satisfying $\mu(\ga)=\mu$. We can write
$$r_\mu=\sum_{k\ge0}r_{k\ga}x^{k\ga},\quad 
a_\mu=\sum_{k\ge1}a_{k\ga}x^{k\ga}$$
and analogously for their Poincar\'e functions.
\end{rmr}

\begin{rmr}
In order to determine $P(a_\mu,v)$, we have to invert
$P(r_\mu,v)$ in the ring $\cQ(v)\tw_{v^2}\pser{x_1,x_2}$.
This can be reduced to the inversion in the usual
ring of power series in the following way.
Consider the map $\cQ(v)\pser x\ar\cQ(v)\tw_{v^2}\pser{x_1,x_2}$
$$x^k\mto v^{-\ang{\ga,\ga}k^2}x^{k\ga},$$
which is obviously a homomorphism of algebras.
The element $P(r_\mu,v)$ is contained in the image of this homomorphism 
and we just have to invert the preimage of $P(r_\mu,v)$
in $\cQ(v)\pser x$.
\end{rmr}

%Let us analyze the second factor of the formula. 
%There is a \la-ring homomorphism 
%$\cQ(q)\ar S$, $q\mto\cL$. 
%Applying it to \ref{} we get
%$$\Exp\left(\frac x{1-\cL}\right)=\sum_{m\ge0}[\infty,m]x^m\big|_{q=\cL}$$
%in $S\pser x$. It follows that
%\begin{multline*}
%\Exp\left(\frac{a_\mu}{1-\cL}\right)
%=\prod_{\al}\Exp\left(\frac{a_\al x^\al}{1-\cL}\right)\\
%=\prod_{\al}\Exp\Big(a_\al\Log\Big(\sum_{m\ge0}[\infty,m]x^{m\al}\big|_{q=\cL}\Big)\Big)
%=\prod_\al\Pow\Big(\sum_{m\ge0}[\infty,m]x^{m\al}\big|_{q=\cL},a_\al\Big)
%\end{multline*}
%in $S\pser{x_1,x_2}$.
  %a_\alpha
\section{On the virtual Hodge polynomial and motive of moduli space}\label{sec:conj}
As we mentioned in introduction, the recursive formula of
Harder, Narasimhan, Desale and Ramanan (see Theorem \ref{thr:HNDR})
for the Poincar\'e polynomials of the moduli spaces 
$\lM(n,d)$ (with coprime $n$ and $d$) of stable bundles on a curve 
was extended by Earl and Kirwan \cite{EarlKirw1}
to the recursive formula for the Hodge numbers of $\lM(n,d)$ for coprime $n$ and $d$.
Let us recall their result. 
We endow the semigroup $\Ga=\cN^*\xx\cZ$ with a total preorder as 
in Section \ref{sec:HN rec}.

\begin{thr}
Let $X$ be a curve of genus $g$ over \cC and, for any $\al\in\Ga$,
let $\lM(\al)$ be the moduli space of stable bundles on $X$ having character
\al. Define the rational functions $R_\al\in\cQ(u,v)$, $\al=(n,d)\in\Ga$, inductively 
by the formula
$$\sum_{\la\in\lP_\al^i}(uv)^{-\sum_{i<j}\ang{\la_i,\la_j}}R_{\la_1}\dots R_{\la_{l(\la)}}
=(u^nv^n-1)\prod_{i=1}^n\frac{(1-u^iv^{i-1})^g(1-v^iu^{i-1})^g}{(1-u^iv^i)^2}.$$
Then the Hodge polynomial of $\lM(n,d)$, for coprime $n$ and $d$, equals
$$(uv-1)R_{(n,d)}(u,v).$$
\end{thr}

\begin{rmr}
The formula in \cite{EarlKirw1} is actually slightly different. Namely, they define
the rational functions $F_\al(u,v)$, $\al\in\Ga$, inductively by the
formula
$$\sum_{\la\in\lP_i}(uv)^{-\sum_{i<j}\ang{\la_j,\la_i}}F_{\la_1}\dots F_{\la_{l(\la)}}
=(1-u^nv^n)\prod_{i=1}^n\frac{(1-u^iv^{i-1})^g(1-v^iu^{i-1})^g}{(1-u^iv^i)^2}$$
and prove that the Hodge polynomial of $\lM(n,d)$, for coprime $n$ and $d$, equals
$(1-uv)F_{(n,d)}(u,v)$. Using the Poincar'e duality, these two formulas 
can be shown to be equivalent.
\end{rmr}

Using the Zagier formula (see Theorem \ref{thr:zagier formula}), we get an explicit formula

\begin{crl}\label{crl:hodge,zagier}
Let $X$ be a curve of genus $g$ over \cC and, for any $\al\in\Ga$,
let $\lM(\al)$ be the moduli space of stable bundles on $X$ having character
\al. Define the rational functions $R_\al\in\cQ(u,v)$, $\al=(n,d)\in\Ga$, 
by the formula
\begin{align*}
R_\al(u,v)
&=\sum_{\stackrel{n_1,\dots,n_k>0}{\sum n_i=n}}
(uv)^{(g-1)\sum_{i<j}n_in_j}\Psi_{n_*,d}(uv)P_{n_1}\dots P_{n_k},
\end{align*}
where
$$P_n(u,v)=(u^nv^n-1)\prod_{i=1}^{n}\frac{(1-u^iv^{i-1})^g(1-v^iu^{i-1})^g}{(1-u^iv^i)^2},$$
$$\Psi_{n_*,d}(t)=\prod_{i=1}^{k-1}\frac
{t^{(n_i+n_{i+1})\set{(n_1+\dots+n_i)d/n}}}
{1-t^{n_i+n_{i+1}}}.$$
Then the Hodge polynomial of $\lM(n,d)$, for coprime $n$ and $d$, equals
$$(uv-1)R_{(n,d)}(u,v).$$
\end{crl} 

We can formulate now the conjectural formula for the virtual Hodge
polynomials of the moduli spaces $\lM(n,d)$.
Denote by $\cQ(u,v)\tw\pser{x_1,x_2}$
the ring obtained from $\cQ(u,v)\pser{x_1,x_2}$ by changing the multiplication
$$x^\al\circ x^\be=(uv)^{-\ang{\al,\be}}x^{\al+\be}.$$

\begin{conj}
With the notations as above, let $A_\al\in\cQ[u,v]$ be the virtual
Hodge polynomial of the moduli space $\lM(\al)$, $\al\in\Ga$. For any $\mu\in\cQ$,
define the generating functions $A_\mu=\sum_{\mu(\al)=\mu}A_\al x^\al$,
$R_\mu=1+\sum_{\mu(\al)=\mu}R_\al x^\al$ in $\cQ(u,v)\pser{x_1,x_2}$.
Then we have
$$R_\mu\circ\Exp\left(\frac{A_\mu}{1-uv}\right)$$
in $\cQ(u,v)\tw\pser{x_1,x_2}$.
\end{conj}

Let $S$ be the Grothendieck ring of motives over \cC, completed
and localized in some appropriate way.
It is a \la-ring, with \si-operations
given by symmetric products \cite{Getz1}. 
For any $\al=(n,d)\in\Ga$, we can define the
motive $M_\al$ (respectively, $R_\al$) of the stack of
vector bundles (respectively, stack of semistable vector bundles) on $X$ having
Chern character \al.
We define $A_\al$ to be the motive of $\lM(\al)$.
It is proved in \cite{BehrDhil1} that
$$M_\al=\Big(Z(C,t)(1-t)(1-\cL t)\Big)\Big|_{t=1}\cL^{(n^2-1)(g-1)}\prod_{i=2}^nZ(C,\cL^{-i}),$$
where $Z(C,t)$ is a motivic zeta-function of $C$, given by $\Exp([C]t)$.
The following result is due to del Ba{\~n}o \cite{Bano1},
although he uses a different terminology

\begin{thr}
For any $\al=(n,d)\in\Ga$, we have
$$M_\al=M_n=\sum_{\la\in\lP_\al^i}\cL^{-\sum_{i<j}\ang{\la_i,\la_j}}
R_{\la_1}\dots R_{\la_{l(\la)}}$$
and
$$R_\al
=\sum_{\stackrel{n_1,\dots,n_k>0}{\sum n_i=n}}
\cL^{(g-1)\sum_{i<j}n_in_j}\Psi_{n_*,d}(\cL)M_{n_1}\dots M_{n_k}.$$
\end{thr}

We can formulate now a conjecture for the motives of the moduli
spaces of stable bundles. Denote by $S\tw_\cL\pser{x_1,x_2}$
the ring obtained from $S\pser{x_1,x_2}$ by changing the multiplication
$$x^\al\circ x^\be=\cL^{-\ang{\al,\be}}x^{\al+\be}.$$

\begin{conj}
We have
$$R_\mu\circ\Exp\left(\frac{A_\mu}{1-\cL}\right)$$
in $S\tw_\cL\pser{x_1,x_2}$, where 
$R_\mu=1+\sum_{\mu(\al)=\mu}R_\al x^\al$ and
$A_\mu=\sum_{\mu(\al)=\mu}A_\al x^\al$. 
\end{conj}
  %hodge

\bibliography{../tex/fullbib}
\bibliographystyle{../tex/hamsplain}
\end{document}